\RequirePackage{cmap}
\documentclass[12pt, a4paper]{article}

\usepackage[utf8]{inputenc}
\usepackage{ amssymb, amsthm }
\usepackage{ amsmath }

\usepackage[T2A]{fontenc}

\DeclareMathOperator{\conv}{conv}
\DeclareMathOperator{\sign}{sign}
\DeclareMathOperator{\Mat}{Mat}
\DeclareMathOperator{\Hess}{Hess}
\newtheorem{theorem}{Theorem}

\newtheorem{lemma}{Lemma}
\theoremstyle{definition}
\newtheorem{defin}{Definition}

\author{Aliaksandr Yuran\footnote{The HSE university. \newline 
		The work was supported by the research assistanship grant of the Faculty of mathematics, HSE.}}
\title{Newton Polytopes Of Nondegenerate Quadratic Forms}

\begin{document}
	\maketitle
	\abstract{We characterise Newton polytopes of nondegenerate quadratic forms and Newton polyhedra of Morse singularities.}

\section {Introduction}

Newton polytopes are objects that can be associated to any polynomial or analytical function. They are defined as follows:

\begin{defin}
	Let  $p=\sum_{k\in \mathbb Z^n}  a_k \mathbf z^k$ be a polynomial of $n$ complex variables where $\mathbf z^k=z_1^{k_1} z_2^{k_2} \dots z_n^{k_n}$. The Newton polytope of $p$ is the convex hull $N(p)=\conv\{k \in \mathbb Z ^n \mid a_k \ne 0\} \subset \mathbb R^n$.
	
	If $p=\sum_{k\in \mathbb Z_{+}^n}  a_k \mathbf z^k$ is an analytic function, then the Newton polytope is defined as $\conv\{k+l\in \mathbb R^n \mid k\in \mathbb Z^n, a_k \ne 0, l \in \mathbb R_{\ge 0} ^n\} $. 
	
\end{defin}

\begin{defin}
	Denote the space of all polynomials from $\mathbb{C}[z_1, \dots, z_n]$ whose Newton polytopes lie inside the polytope $M$ as $\mathbb C^M$. 
\end{defin}
The Newton polytope $N(f)$ carries important information about $f$ (a classical result of this type is Kouchnirenko Theorem \cite{MilnorNumber}, see Theorem \ref{Kouchnirenko}). This article is devoted to a natural question: "How do Newton polytopes of nondegenerate quadratic forms / Morse singularities look like?" We prove the following results:
\begin{enumerate}
	\item If $O=\left ( \frac{2}{n},\dots,\frac{2}{n} \right )\in M$, then a generic quadratic form $B\in\mathbb C^M$ is nondegenerate. Otherwise, $B$ is degenerate.
	\item A generic function $f\in \mathbb C^M$ with $f(0)=0$ and $d_0 f=0$ has a Morse singularity at 0 if $O\in M$. Otherwise, the singularity is not Morse.
\end{enumerate}

\textbf{Comment.} 
One direction of the second statement above follows from deep results of A. N. Varchenko on complex oscillating integrals. Namely, if the singularity $f$ is Morse, then the point $O$ belongs to the Newton diagram. The results of Varchenko relate a certain invariant of the singularity $f$ (so called complex oscillation index $I$, see \cite[\S 13.1.5]{AVG}) to the remoteness $R$ of the Newton diagram (i.e.the smallest $r$ such that the point $(-\frac{1}{r},\dots,-\frac{1}{r})$ belongs to the diagram). Since the complex oscillation index attains its minimal possible value $-n/2$ only for the Morse singularity (see \cite[Corollary in \S13.3.3]{AVG}), and the remoteness takes its smallest possible value only for Newton diagrams containing $O$, the inequality $I \ge R$ (see \cite[\S 13.1.7, Theorem 13.2]{AVG}) implies the sought result. Our proof of this implication is much more elementary and constructive (see the comment to Theorem \ref{Th2} below).


\textbf{Comment}. This problem arose, in particular, in connection with the monodromy conjecture, and was solved in the dimensions up to 4 by technical casework  in \cite[Lemma 4.9]{alex2013monodromy}.

	\textbf{Structure of the paper.} We characterise Newton polytopes of quadratic forms in Theorem \ref{Th1} (section 3.1), Theorem \ref{Th2} (sections 3.2--3.3) and the ones of singularities in Theorem \ref{Th3} (section 4). 
	
	The proof of the Theorem \ref{Th2} consists of 2 steps. The first step is the construction of a zigzag of nonzero entries in the matrices from $\mathbb C^M$; the second one is the proof that the existence of such a zigzag is enough for a generic quadratic form to have nonzero determinant. We prove both steps in the Section 3.2. Another way to make the step 1 is described in the Section 3.3.
	
	\textbf{Acknowledgement.} This article was written under the scientific supervision of Alexander Esterov.

\section{Notation}
\begin{defin}
	Let $A_i$ (where $i\in \{1, 2, \dots,n\}$) be the point in $\mathbb R^n$ whose $i$-th coordinate is 2 while the others are 0. We denote the convex hull of all the points $A_i$ by $2\Delta$.
		
	Then any of the integer points in $2\Delta$ can be expressed as $A_{ij}=A_{ji}=\frac{A_i+A_j}{2}$. 
	
	Let $O=\left ( \frac{2}{n},\dots,\frac{2}{n} \right )$ be the barycenter of $2\Delta$.
\end{defin}

\newcommand {\tB}{\tilde B}

\begin{defin} 
	Suppose that we have a lattice polytope $M \subset 2\Delta$. Then define stencil as the matrix $\tilde {B}(M) \in \Mat_n(\{0,1\})$ the entries of which are $\tilde {B}_{ij}=1$ if and only if $A_{ij} \in M$.
	
\end{defin}
	
	\textbf{Comment.} It is easy to see that the stencil $\tilde {B}(M)$ encodes the polytope $M$. For $M\subset 2\Delta$ the space $\mathbb C^M$ consists exactly of the quadratic forms $B$ for which the matrix entries $B_{ij}$ equal 0 if $\tB(M)_{ij}$ equal 0.

	
\section{Main theorems} \label{MainSection}

 \subsection{Newton polytopes of nondegenerate quadratic forms}

\begin{theorem} \label{Th1}
	If $B$ is a nondegenerate quadratic form, then $O\in N(B)$.
\end{theorem}

\begin{proof}
	
	Remember that if $B$ is nondegenerate, then \[\det B=\sum_{\sigma \in S_n} \sign(\sigma) B_{1,\sigma(1)} \dots B_{2,\sigma(2)}B_{n,\sigma(n)} \neq 0.\]
	
	Thus,  $B_{1,\sigma(1)} \dots B_{2,\sigma(2)}B_{n,\sigma(n)} \neq 0$ for some $\sigma \in S_n$. Consequently, each of the points $A_{i,\sigma(i)}$ belongs to $N(M)$. Now we can express $O$ as a convex combination of the points from $N(M)$ in the following way: 
	\[ \sum_{i=1}^n \frac{1}{n} \cdot A_{i,\sigma(i)} = \frac{1}{2n} \sum_{i=1}^n A_i + A_{\sigma(i)}= \frac{1}{n} \sum_{i=1}^n A_i=O.\]

\end{proof}

\subsection{Quadratic forms with given Newton polytopes}

	

\begin{theorem} \label{Th2}
	Suppose that a lattice polytope $M \subset 2\Delta$ contains $O$; then a generic quadratic form $B\in \mathbb C^M$  (i.e. $B$ for which $N(B)\subset M$) is nondegenerate.
\end{theorem}

\textbf{Comment.} Our first proof provides a way to find the zigzag of nonzero entries in $B$ if $O\in N(B)$. Our second proof provides a way to construct a hyperplane separating $O$ from $N(B)$ given a large enough rectangle of zeroes in the matrix $B$ (upon a permutation of rows and columns).

We begin the proof of the theoren \ref{Th2} with the following

	\begin{defin}
		A lattice polytope $M\subset 2\Delta$ is called minimal if it contains $O$ and every lattice polytope $M' \subset M$ containing $O$ is equal to $M$.
	\end{defin}

	Now we prove several lemmata. $M$ is a minimal polytope and $\{C_1, C_2, \dots, C_k\}$ is the set of all its lattice points in the first four of them.  
	\begin{lemma}
		None of the stencil rows $\tB(M)$ consists of zeroes.
	\end{lemma}
	
		\begin{proof}
			If the $i$-th row consists of zeroes, then the $i$-th coordinate of any point of $M$ is zero, while the $i$-th coordinate of $O$ is not.
		\end{proof}
	
	\begin{lemma}	
		The polytope $M$ contains at most one of the points $ A_1, A_2, \dots, A_n $. In other words, there is at most one entry on the diagonal of the stencil $\tB(M)$ that equals 1.
	\end{lemma}
		\begin{proof}
			 Let $C_1=A_1$ and $C_2=A_2$. Since $O\in M$ we can express this point as a convex combination of $\{C_i\}$ in the following way: $O=\alpha_1 A_1+ \alpha_2 A_2 + \sum_{i=3}^k \alpha_i C_i$, where $\alpha_1+...+\alpha_k=1$. 
			By symmetry, let $\alpha_1\le\alpha_2$. 
			But then since $A_1+A_2=2A_{1,2}$ we have $O=2\alpha_1A_{1,2} + (\alpha_2-\alpha_1) A_2 + \sum_{i=3}^k \alpha_i C_i$, so we obtain a polytope $M'=\conv \{A_{1,2}, A_2; C_3, \dots, C_k\} \subsetneq M$ that contains $O$.
		\end{proof}
	
	\begin{lemma}
		$M$  is a simplex (possibly not of maximal dimension) that contains at most $n$ lattice points. Thus, $\tB(M)$ has at most $2n$ nonzero entries.
	\end{lemma}
	\begin{proof}
		
		The statement easily follows from Carath\'eodory theorem, which states that if the convex hull of a set $K\subset \mathbb R^{n-1} $ contains a point $P$, then there is a subset $K' \subset K$ with cardinality no more than $n$ such that $\conv K'$ is a simplex containing $P$.
		
		Remember that $M$ lies in the $(n-1)$-dimensional subspace of $\mathbb R^n$ containing all $A_i$-s. Replacing $K$ by $M\cap \mathbb Z^n$ and $P$ by $O$ we obtain a containing the point $O$ polytope $M'\subset M$ with at most $n$ vertices. Now $M'=M$ since $M$ is minimal.
		
		The only remaining issue is to verify that all the lattice points of $M$ are its vertices. There are only two possible representations of $A_{ij}$ as a convex combination of lattice points from $2\Delta$, namely $A_{ij}=1\cdot A_{ij}$ and $A_{ij}=\frac{1}{2}A_i+\frac{1}{2}A_j$. Thus, if $A_{ij} \in M$, then either $A_{ij}$ is a vertex of $M$ or both $A_i$ and $A_j$ are vertices of $M$. The second variant contradicts the Lemma 2, so the conclusion follows.

	\end{proof}

\begin{defin}
	We call the vertex $A_{ij}$ of a polytope $M \subset 2 \Delta$ special if $A_{il}\notin M$ for any $l \ne j$. 
	
	In fact, since $A_{ij}=A_{ji}$ we should verify that special vertices are defined properly. This is done by the following   
\end{defin}	

	\begin{lemma}
		Suppose $A_{ij}$ is a special vertex of $M$; then $A_{lj}\notin M$ for any $l \ne j$.   
	\end{lemma} 

	\begin{proof}
		If $i=j$ there is nothing to prove since $A_{ij}=A_{ji}$.
			
		Now assume $i \ne j$. Since $O\in M$, we can express it as a convex combination of vertices of $M$ (and possibly $A_{lj}$ for several $l$):
		\[O=\sum_{l = 1}^n \alpha_l A_{lj} +\sum_{l=n+1}^{k'} \alpha_l C_l\]
			
		Here $C_l$ are the vertices of $M$ the $i$-th and $j$-th coordinates of that are 0. 
			
		Let us look at the $i$-th coordinate of $O$:
			
		\[\frac{2}{n}=(O)_{i} =\alpha_i \cdot (A_{ij})_i+\sum_{l \in 1\dots n, l \ne i} \alpha_l \cdot (A_{lj})_i +\sum_{l=n+1}^{k'} \alpha_l \cdot (C_l)_i=\alpha_i\cdot 1\]
			
		Now let us look at the $j$-th one:
			
		\[
		\frac{2}{n}= (O)_{j} =
		(\alpha_i\cdot 1)
		+(\alpha_j+\sum_{l \in 1\dots n, l \ne i} \alpha_l\cdot 1)
		+(0)
		=\frac{2}{n}+(\alpha_j+\sum_{l \in 1\dots n, l \ne i} \alpha_l)
		\]
			
		So, we obtain that 
		for each $l \in 1\dots n, l \ne i$ we have $ \alpha_l=0 $ using the nonnegativity of $ \alpha_l $.   
		And now 	\[O=\alpha_i A_{ij} +\sum_{l=n+1}^{k'} \alpha_l C_l\]
			
		Thus, using the minimality of $M$, we obtain $A_{lj} \notin M$ for any $l \ne i$. This is equivalent to the Lemma's statement.

	\end{proof}

\begin{lemma} \label{reduction_special}
		Suppose a polytope $M \subset 2\Delta$ has vertices $C_1, C_2, C_3, \dots, C_k$ whose vertex $C_k=A_{ij}$ is special. Then $M$ is minimal if and only if the polytope $M'=\conv\{C_1, \dots, C_{k-1}\}\subset 2\Delta' $ is minimal in the simplex $2\Delta'=2\Delta \cap \{x_i=0\} \cap \{x_j=0\}$.
	\end{lemma}
		
	\begin{proof}
		Suppose that $i\ne j$ (the case $i=j$ is analogous). 
		
		First, observe that $M$ is a simplex if and only if $M'$ is a simplex. 
		
		Let $M\subset 2\Delta$ be a minimal polytope. Lemma 3 implies that $M$ is a simplex, so we may represent $O$ as the \textbf{unique} convex combination of the vertices $C_1, \dots C_k$ as follows: $O=\sum_{l=1}^{k-1} \alpha_l C_l + \alpha_k A_{ij}$ where $\alpha_l\ne 0$ for all $l$ (if $\alpha_l= 0$, then $M$ is not minimal). By Lemma 4 for any $l < k$ the vertices $C_l$ have the $i$-th and $j$-th coordinates equal to 0. Thus, looking at the $i$-th affine coordinate of $O$ we see $\frac{2}{n}=(O)_i=\alpha_k$, so $\sum_{l=1}^{k-1} \alpha_l=\frac{n-2}{n}$.
		
		Now let $O'$ 
		be the barycenter of $2\Delta'$. It is easy to verify that $O'=\sum_{l=1}^{k-1} \frac{n}{n-2}\alpha_l C_l=\sum_{l=1}^{k-1} \beta_l C_l$ by counting all the affine coordinates of $O'$. That is, $O' \in M'$. Since none of the $\beta_l$ equals to 0, we conclude that $O'$ is strictly inside of the simplex $M'$, so $M'$ minimal.
		
		Analogically, if $M'$ is minimal, then  $O'=\sum_{l=1}^{k-1} \beta_l C_l$, so $O=\sum_{l=1}^{k-1} \frac{n-2}{n}\beta_l C_l + \frac{2}{n} A_{ij}$ lies strictly inside of the simplex $M$. 
		
	\end{proof}

	\begin{proof}[Proof of the theorem \ref{Th2}]

\textbf{Step 1.}

	Suppose that $M\subset 2\Delta \subset \mathbb R^n $ is a minimal polytope; then there exists a permutation $\sigma \in S_n$ of indices such that $\tB(M)_{1,\sigma(1)}=\dots=\tB(M)_{n,\sigma(n)}=1$. 

	We will prove this by induction on the number of special vertices of $M$.
	
	\textbf{Base.} Assume that $M$ has no special vertices.  Any row of $\tB(M)$ contains at least one nonzero element by Lemma 1. Though, if the $i$-th row contains exactly one nonzero $\tB(M)_{ij}$ then $A_{ij}\in M$ is unique. Thus, every row contains at least 2 nonzero elements, but  there are at most $2n$ nonzero entries in $\tB(M)$. So, each row (and column) of $\tB(M)$ contains exactly two entries equal to 1.
	
	Consider the graph $G$, the vertices of that are couples $(i,j)$ such that $\tB(M)_{ij}=1$. An edge connects $(i_1,j_1)$ and $(i_2,j_2)$ if and only if $i_1=i_2$ (a vertical edge) or $ j_1=j_2 $ (a horizontal edge). 
	
	The degree of each vertex is two, so the graph consists of several cycles. The edges of the two types alternate in each cycle, so every cycle has an even length. Therefore, we may choose a set of exactly $n$ vertices that are pairwise disjoint. 
	
	We got the set of pairs $\{(i_1,j_1),\dots,(i_n,j_n)\}$. Since the first components of the pairs are distinct as well as the second ones, the desired permutation can be defined as $\sigma(i_l)=j_l$. 
	
	\textbf{Induction step.}
	Let the vertex $A_{ij}$ of $M$ be special. Then $\tB_{ij}=\tB_{ji}=1$. Without loss of generality, we may assume that either $i=j=n$, or $i=n$ and $j=n-1$. Then $M'$ is minimal but has one special vertex less. (Here we use the notation from \ref{reduction_special}). If we erase the $i$-th and $j$-th lines and columns of the matrix $\tB(M)$ we obtain the stencil $\tB(M')$. By the induction hypothesis we have a permutation $\sigma'\in S^{min\{i,j\}-1}$ for that $\tB(M)_{i,\sigma'(i)}=\tB(M')_{i,\sigma'(i)}=1$ holds.
	
	Now \[\sigma(n):=
	\begin{cases}
	\sigma'(n), n \le min\{i,j\}-1;\\
	i, n=j;\\
	j, n=i.\\
	\end{cases}\]
	

\textbf{Step 2.}
	Since  $M$ contains $O$, there exists a minimal $M'\subset M$, we have constructed a permutation $\sigma_0$ such that  $\tB(M)_{i,\sigma_0 (i)}=1$ for any $i$.  We want to verify that the polynomial $\det (B)$ is not identically zero on $\mathbb C^M$. Let us look at the expansion
	$\det B=\sum_{\sigma \in S_n} \sign(\sigma)B^\sigma=$
	
	$\sum_{\sigma \in S_n} \sign(\sigma) B_{1,\sigma(1)}B_{2,\sigma(2)} \dots B_{n,\sigma(n)}$. Now we are going to show that for every monomial $B^{\sigma_1}$ that is equal to $B^{\sigma_0}$, we have $\sign (\sigma_0)=\sign (\sigma_1)$. 

	Let us represent $\{1,2, \dots n\}$ as the union of the subsets $E=\{i\in \{1\dots n\}\mid \sigma_0(i)=\sigma_1(i)\}$ and $F=\{1 \dots n\} \setminus E$. Since $B^{\sigma_0}=B^{\sigma_1}$, the sets of \textbf{unordered} pairs $\{\{i,\sigma_0(i)\}\mid i\in 1\dots n\}$ and $\{\{i,\sigma_1(i)\}\mid i\in 1\dots n\}$ are equal. Using the fact that for any $i\in E$ the equality $\{i,\sigma_0(i)\}=\{i,\sigma_1(i)\}$ holds, we obtain that for any $i\in F$ there exists a different from $i$ number $j\in F$ such that $\{i,\sigma_0(i)\}=\{j,\sigma_1(j)\}$.  That is, $\sigma_0(i)=j$ and $\sigma_1(j)=i$, so $\sigma_1(\sigma_0(i))=i$. We have shown that $\sigma_1=\sigma_0^{-1}$ on $F$ and $\sigma_1=\sigma_0$ on $E$. Now let 
	\[ \sigma_2(n):=
	\begin{cases}
	\sigma_1(n), n \in F;\\
	n, n \in E.\\
	\end{cases}\]
	
	Evidently, $\sigma_1=\sigma_0\sigma_2^2$, so $\sign \sigma_1=\sign \sigma_0$. This fact shows that the coefficient of the monomial $B_{1,\sigma_0(1)} \dots B_{n,\sigma_0(n)}$ in the polynomial $\det B$ is nonzero (since it has the same sign as the permutation $\sigma_0$). Thus, $\det B$ is also not an identical zero.
	
	Therefore, generic quadratic form from $\mathbb C^M$ is nondegenerate, since the subvariety $\{\det B=0\} \subset \mathbb C^M$ of degenerate quadratic forms does not coincide with $\mathbb C^M$.
	
\end{proof}

\subsection{Approach to the Theorem \ref{Th2} Using K\"onig Theorem}

We will present a proof of Theorem \ref{Th2} without using the notion of a minimal polytope with the following

\begin{theorem}[K\"onig]
	Let $B$ be a $n \times n$ matrix consisting of zeroes and ones. Then the following conditions are equivalent:
	\begin {itemize}
	\item[(1)] There exists a permutation $\sigma\in S_n$ such that $B_{1,\sigma(1)}=\dots=B_{n,\sigma(n)}=1$.
	\item[(2)] Suppose that $I$ and $J$ are subsets of $\{1,\dots,n\}$ with the following property:  $(B_{ij}=1)$ $\Rightarrow$ $(i\in I$ or $j\in J)$; then $|I|+|J|\ge n$.
	\end{itemize}
\end{theorem}

\textbf{Comment.} This theorem is usually formulated in terms of the bipartite graphs with the adjacency matrix $B$.
	
The main statement of Step 1 of the Theorem \ref{Th2} now can be replaced with the following:

\textbf{Version of Step 1.}
\textit{Suppose that $O \in M\subset 2\Delta$; then there exists a permutation $\sigma \in S_n$ of indices such that $\tB(M)_{1,\sigma(1)}=\dots=\tB(M)_{n,\sigma(n)}=1$.}

\begin{proof} 
	We need to verify the property (1) for the matrix $\tB=\tB(M)$. It is enough to verify the property (2).
	
	Let $I$ and $J$ be the sets with the described in K\"onig theorem property. Assume that $|I|+|J| < n$. We prove that $O\notin M$.
	
	Consider the halfspace $\Gamma$ given by the inequality 
	
	\[ \label{Gamma}
	\sum_{l\in I} x_{l} + \sum_{l\in J} x_{l} \ge 2. \eqno{(*)}
	\]
	
	First, we prove that $M\subset \Gamma$. If $A_{ij}$ is a vertex of $M$, then $\tB(M)_{ij}=\tB(M)_{ji}=1$, so using the property (2) we get that $(i\in I$ or $j\in J)$ and $(j\in I$ or $i\in J)$. This is equivalent to the fact that $i,j\in I$, or $i,j\in J$, or $i\in I\cap J$, or $j\in I\cap J$. The fact that $A_{ij} \in \Gamma$ can be verified easily by substituting the coordinates of $A_{ij}$ into ($*$) in each of the four cases. In the first case the first sum in ($*$) is at least 2, in the second case the second sum is at least 2, and in the last two cases both sums are at least 1.
	
	   Note that $O \notin \Gamma$ since $\sum_{l\in J} O_{l} + \sum_{l\in I} O_{l} = \frac{2}{n} \cdot (|I|+|J|) < 2$. Therefore, $O\notin M$.
\end{proof} 

The Step 2 of the proof retains.

\section{An Application to Singularities}

Let $f: \mathbb C^n \to \mathbb C$ be an analytical function. Suppose that $f(0)=0$ and $d_0f=0$, that is, $f$ has a singularity at 0. We derive the following theorem as a consequence of the previous results:
\begin{theorem} \label{Th3}
	If $O \notin N(f)$ then the singularity is not Morse.
	
	If a lattice (infinite) polytope $M\subset \mathbb R^n_{\ge 0}$ contains the point $O$, then a generic $f$ with $d_0f=0$ and $N(f) \subset M$ has a Morse singularity at 0.
\end{theorem}  

\textbf{Comment.} The exact condition of being \textit{generic} here is that if $B$ is a generic quadratic form from $\mathbb C^{M\cap 2\Delta}$, then \textbf{any} $f=B(\mathbf{x})+o(\mathbf{x}^2)$ has a Morse singularity.

\begin{proof} Let $f=B(\mathbf{x})+o(\mathbf{x}^2)$, where $B$ is a quadratic form. Since $N(B)=N(f) \cap 2\Delta$, it follows that $O\in N(f)$ is equivalent to $O\in N(B)$.  Thus,  if $O\notin N(f)$, then $0=\det B=\Hess(f)$ and $f$ has a Morse singularity. If $O\in M$ then $\det B \ne 0$ for almost every $B\in \mathbb C ^{M \cap 2\Delta}$, that is for almost every $f$ for that $N(f) \subset M$ holds.
\end{proof}

It is interesting to think of this theorem in the context of the following result:

\begin{theorem}[Kouchnirenko, \cite{MilnorNumber}] \label{Kouchnirenko}
	Let $f$ be a generic analytical function of $n$ variables with a singularity at zero where $f(0)=0$. Denote $M=\mathbb Z^n_{\ge 0} \setminus N(f)$. Suppose that $M$ is bounded. Consider the $\binom{i}{n}$ intersections of $M$ and all the $i$-dimensional coordinate subspaces. Denote the sum of their $i$-dimensional volumes as $V_i$. Then the Milnor number of $f$ can be counted as $\mu (f)=n!\cdot V_n-(n-1)!\cdot V_{n-1}+\dots+(-1)^{n-1}V_1+(-1)^n$.
	
	The statement of the theorem holds for the functions $f$ such that for any face $\varGamma$ of $N(f)$ there are no points $x\in(\mathbb C\setminus \{0\})^n$ for which $f^\varGamma(x)=0$ and $df^\varGamma(x)=0$. Here $f^\varGamma$ stands for the sum of all monomials in $f$ corresponding to the lattice points in $\varGamma$.
\end{theorem}

\textbf{Comment 1.} One might expect that Theorem \ref{Th3} could be proved using the Kouchnirenko Theorem subsituting $\mu=1$. However, the author does not know how do that and it would be interesting to obtain such a proof. Possibly, that proof could be much more difficult since the Kouchnirenko Theorem does not give a positive formula for $\mu$.

\textbf{Comment 2.} The statement of Theorem \ref{Th3} is valid for all functions $f$ that are generic in the sense of Kouchnirenko Theorem.


\end{document}